\documentclass[12pt,a4paper,reqno]{amsart}
\usepackage{fullpage}
\usepackage{hyperref}
\usepackage{url}
\usepackage{newtxtext,newtxmath}
\usepackage{amsmath}

\newtheorem{theorem}{Theorem}

\newtheorem{lemma}[theorem]{Lemma}

\newtheorem{remark}{Remark}

\DeclareMathOperator{\lead}{lead}
\newcommand{\Etilde}{\widetilde{\mathcal{E}}}
\newcommand{\Stilde}{\widetilde{S}}
\newcommand{\dx}{\mathrm{d}}

\newcommand{\N}{\mathbb{N}}

\newcommand{\Z}{\mathbb{Z}}
\newcommand{\e}{\mathrm{e}}

\newcommand{\eps}{\varepsilon}
\newcommand{\Odip}[2]{\mathcal{O}_{#1}\!\left(#2\right)\mathchoice{\!}{}{}{}}
\newcommand{\Odig}[1]{\mathcal{O}\!\Bigl(#1\Bigr)\mathchoice{\!}{}{}{}}

\newcommand{\Odim}[1]{\mathcal{O}\bigl(#1\bigr)}
\newcommand{\Odi}[1]{\Odip{}{#1}}

\allowdisplaybreaks

\title{On the average number of representations of an integer
as a sum of polynomials computed at prime values}

\author{Alessandra Migliaccio and Alessandro Zaccagnini}

\address{
  AM: Dipartimento di Matematica e Informatica,
  Universit\`a degli Studi di Ferrara, Ferrara, Italy.
  email: \texttt{alessand.migliaccio@edu.unife.it}.
  AZ: [Corresponding Author]
  Dipartimento di Scienze Matematiche, Fisiche e Informatiche,
  Universit\`a di Parma, Parco Area delle Scienze, 53/a, 43124 Parma, Italy
  email: \texttt{alessandro.zaccagnini@unipr.it}}

\subjclass[2010]{Primary 11P32. Secondary 11P55, 11P05}

\date{\today}

\begin{document}

\begin{abstract}
We study the average number of representations of an integer $n$ as
$n = \phi(n_{1}) + \dots + \phi(n_{j})$, for polynomials
$\phi \in \Z[n]$ with $\deg(\phi) = k\ge 1$, $\lead(\phi) = 1$,
$j \ge k$, where $n_{i}$ is a prime power for each $i \in \{1, \dots, j\}$.
We extend the results of Languasco and Zaccagnini in \cite{bib12}, for
$k=3$ and $j=4$, and of Cantarini, Gambini and Zaccagnini in
\cite{bib1}, where they focused on monomials $\phi(n) = n^k$, $k\ge 2$ and
$j=k, k + 1$.
\end{abstract}

\maketitle

\section{Introduction}

Our goal is the study of the average number of representations of an
integer as the sum of values of a polynomial with integer coefficients
computed at prime powers.
More precisely, we will consider the following setting: we take
\[
  \phi\in\mathbb{Z}[n]
  \quad\text{with}\quad
  \phi(n)=\sum_{h=1}^k a_{h}n^h
  \quad\text{and}\quad
  a_{k} = 1.
\]
The choice $a_{k}=1$ simplifies the generalization of previous
results, but we will remove this constraint later; we also assume that
$\phi(0) = 0$ and that $\phi(n) \not \equiv n^k$ identically.
Let us define
\begin{equation}
\label{sec1:1}
  R_{\phi, j}(n)
  =
  \sum_{n=\phi(n_{1})+\dots+\phi(n_{j})}
    \Lambda(n_{1}) \dots \Lambda(n_{j}),
\end{equation}
for $j\ge k$, where $\Lambda(n)$ is the von Mangoldt function defined
as
\[
  \Lambda(n)
  =
  \begin{cases}
    \log p & \text{if $\exists\,t\in\N^*\,:\, n=p^t$} \\
    0 & \text{otherwise}.
  \end{cases}
\]
We study the average number of representations $R_{\phi, j}(n)$ of a
positive integer $n$, as sums of values of this kind of polynomials,
in short intervals of the form $[N+1, N+H]$, for $H\ge 1$ as small as
possible.

The main unconditional results of Cantarini, Gambini and Zaccagnini in
\cite{bib1}, extending earlier work by Languasco and Zaccagnini in
\cite{bib12}, were collected in their Theorems 1.1 and 1.3.
These authors studied the weighted average number of representations
$R_{k}$ and $R'_{k}$ of an integer $n$ as a sum of prime powers,
where, in our notation, $R_{k}(n)=R_{\phi,k+1}(n)$ and
$R'_{k}(n)=R_{\phi,k}(n)$, with $\phi(n)=n^k$.
We prove that estimates similar to those in \cite{bib1} still hold for
averages of $R_{\phi,j}(n)$, up to constants depending on the
polynomial $\phi$.
In contrast to the papers \cite{bib12} and \cite{bib1}, we give a
single statement and proof valid for all $j \ge k$, the most
interesting case being when $j$ is small, that is, $j = k$.
For positive integers $j$ and $k$ it is convenient to write
\[
  \gamma_{k}:=\Gamma\left(1+\frac{1}{k}\right)
  \quad \text{and}\quad
  \gamma_{k,j}:=\Gamma\left( \frac{j}{k}\right),
\]
where $\Gamma$ is the Euler Gamma function.
Throughout the paper, implicit constants may depend on the polynomial
$\phi$, or its degree $k$, and on the integer $j$.
Throughout the paper we write $L = \log(N)$ for brevity.
Finally, let
\begin{equation}
\label{sec1:2}
  A(N; C)
  =
  \exp\left(C\left(\frac{\log N}{\log\log N}\right)^{1/3}\right).
\end{equation}

\begin{theorem}
\label{thm-unc}
Let $j \ge k \ge 2$ be fixed integers.
Then, for every $\eps > 0$, there exists a constant $C = C(\eps) > 0$,
independent of $\phi$ and $j$, such that
\[
  \sum_{n=N+1}^{N+H} R_{\phi,j}(n)
  =
  \frac{\gamma_{k} ^j}{\gamma_{k,j}}\cdot HN^{(j-k)/k}
  +
  \Odi{H N^{(j - k) / k} A(N; -C)},
\]
as $N\to\infty$, uniformly for $N^{1-13/15k+\eps} < H < N^{1-\eps}$.
\end{theorem}

The improvement on the range of uniformity for $H$ in our
Theorem~\ref{thm-unc} is due to the recent result of Guth and Maynard
in \cite{bib4} on the density of the zeros of the Riemann zeta-function,
which would also affect all previous unconditional results in the
papers \cite{bib12} and \cite{bib1}.

\subsection{Outline of the proof}

We briefly sum up the strategy used in \cite{bib1} and \cite{bib12}.
We consider $j \ge k \ge 2$ fixed.
For $\phi$ as above, we introduce the exponential sum
\begin{equation}
\label{sec1:3}
  \Stilde_{\phi}(\alpha)
  =
  \sum_{n\ge 1} \Lambda(n) \e^{-\phi(n)/N} \e(\phi(n)\alpha)
  =
  \sum_{n\ge 1} \Lambda(n) \e^{-\phi(n)z},
\end{equation}
where $\e(\alpha) = \e^{2 \pi i \alpha}$,
$z = \displaystyle{\frac{1}{N}-2\pi i\alpha}$, and the finite sum
\[
  U(\alpha, H)
  =
  \sum_{m=1}^H \e(m\alpha), \qquad\text{for}\quad H\le N.
\]
Thus, we can rewrite $R_{\phi,j}(n)$ as
\[
  R_{\phi,j}(n)
  =
  \e^{n/N}
  \int_{-1/2}^{1/2} \Stilde_{\phi}(\alpha)^j \e(-n\alpha) \, \dx \alpha,
\]
by definition $\eqref{sec1:1}$ of $R_{\phi,j}(n)$ and by
orthogonality.
Summing $R_{\phi,j}(n)$ over $n$, we obtain that
\[
  \sum_{n=N+1}^{N+H} R_{\phi,j}(n) \e^{-n/N}
  =
  \int_{-1/2}^{1/2}
    \Stilde_{\phi}(\alpha)^j U(H, -\alpha) \, \e(-N\alpha) \, \dx\alpha.
\]
The expected main term comes from a comparatively short arc around 0.
Furthermore, $\Stilde_{\phi}(\alpha)$ is of size
$\gamma_{k}z^{-1/k}$, for ``small" $\alpha$.
Hence, we set $\tau = B H^{-1}$ and decompose the quantity we are
studying as follows:
\begin{align}
\notag
  \sum_{n=N+1}^{N+H} R_{\phi,j}(n) \e^{-n/N}
  &=
  \gamma_{k}^j
  \int_{-\tau}^{\tau} z^{-j/k}U(-\alpha, H) \e(-N\alpha) \, \dx \alpha \\
  &\quad+ \notag
  \int_{-\tau}^{\tau}
    (\Stilde_{\phi}(\alpha)^j-\gamma_{k}^jz^{-j/k})U(-\alpha, H)
    \e(-N\alpha) \, \dx \alpha \\
  &\quad+
  \int_{\mathcal{C}} \Stilde_{\phi}(\alpha)^j U(-\alpha, H) \e(-N\alpha)
    \, \dx \alpha \,
  =:\label{I123}
  \gamma_{k}^jI_{1} + I_{2} + I_{3},
\end{align}
say, where $\mathcal{C} = [-1/2, -\tau] \cup [\tau, 1/2]$ and
$B = N^{2 \eps}$.
The parameter $\tau$ controls the width of the ``major arc" around 0,
where it is easy to obtain a good approximation to $\Stilde_{\phi}$.

In fact, for small $\alpha$ the exponential sum $\Stilde_\phi$
behaves pretty much like
\begin{equation}
\label{sec1:2.5}
  \Stilde_{k}(\alpha)
  =
  \sum_{n \ge 1} \Lambda(n) \e^{-n^k /N} \, \e(n^k \alpha).
\end{equation}
The evaluation of $I_{1}$ is in Lemma~4 of \cite{lan16jnt}.
The bound for $I_{2}$ relies on our Lemma~\ref{lemma-L2} below, which
generalizes Lemma~4 of \cite{bib12} since we have $\Stilde_\phi$ in
place of $\Stilde_k$.
The estimate of $I_{3}$ requires Lemma ~\ref{lemma-L2}, as it is shown in Lemma ~\ref{lemma-tolev}.
Finally, we can conclude, using the same method of \S3 of
\cite{bib12}.

Two final remarks: the lower bound for $H$ in our Theorem~\ref{thm-unc}
is only due to the limitations in Lemma~\ref{lemma-L2}; at all other
places we just need the more natural bound $H > N^{1 - 1 / k + \eps}$.
Our proof does not rely on the generalization of Lemmas~6 and~7 of
\cite{bib12} to our case: indeed, we do not need bounds for the $L^4$
norm of $\Stilde_\phi$, but show that $L^2$ estimates suffice.
Apparently this simplification, which depends on the use of identity
\eqref{sec3:8} below introduced in \cite{bib1}, was overlooked by its
authors.

In contrast to previous papers, we do not have any conditional result
because the proof of Lemma~\ref{lemma-L2} on the whole unit interval
$[-\frac12, \frac12]$ seems to be quite hard: the Riemann Hypothesis
does not seem to help on the complement of $[-\tau, \tau]$.

\section{Lemmas}

It will spare us some repetition to notice at the outset that for
fixed $\alpha \ge 0$ and $k \ge 1$ we have
\begin{align}
\notag
  \int_0^{+\infty} t^\alpha \e^{- t^k / M} \, \dx t
  &=
  \frac1k M^{(\alpha + 1) / k}
  \int_0^{+\infty} u^{(\alpha + 1) / k - 1} \e^{- u} \, \dx u \\
\label{bound-int}
  &\ll_{\alpha, k}
  M^{(\alpha + 1) / k},
\end{align}
as $M \to +\infty$, which can be seen by the obvious change of variable.
We also recall the Prime Number Theorem in the weak form
\begin{equation}
\label{PNT}
  \psi(x)
  =
  \sum_{n \le x} \Lambda(n)
  \sim x
  \qquad\text{as $x \to +\infty$.}
\end{equation}

In order to deal with $I_{3}$, we need a new version of Lemma~3.3 of
\cite{bib1} with $\Stilde_{\phi}$ in place of $\Stilde_{k}$.

\begin{lemma}
\label{bound-S}
We have that $\Stilde_{\phi}(\alpha) \ll_{\phi} N^{1/k}$.
\end{lemma}

\begin{proof}
The proof is achieved by partial summation applied to the exponential
sum defined in \eqref{sec1:3}, using~\eqref{bound-int} with $\alpha = k$
and $M = N$, say, and \eqref{PNT}.
\end{proof}

\begin{lemma}
\label{lem-Sinf}
Let $R=R(N)\ge 1$ and split $\Stilde_{\phi}$ as
\begin{align*}
  \Stilde_{\phi}(\alpha)
  =
  \Stilde_{\phi,R}(\alpha) + \Stilde_{\phi,\infty}(\alpha)
  :=
  \left(\sum_{n\le R} + \sum_{n>R}\right)
    \Lambda(n) \, \e^{-\phi(n)/N} \, \e(\phi(n)\alpha).
\end{align*}
If $R\gg N^{1/k}L$, then $|\Stilde_{\phi,\infty}|\ll_{\phi} N^{1/k}$.
\end{lemma}
\begin{proof}
By partial summation, as in the proof of Lemma~\ref{bound-S}, we see
that $\Stilde_{\phi,\infty} \ll R \e^{-R^k /  N}$, which is much
smaller than $N^{1/k}$ if we choose $R = N^{1/k}L$, say.
\end{proof}

The estimate of the term $I_{2}$ is based on Lemma~4 of \cite{bib12}.
We just remark that with the new bounds due to Guth and Maynard
\cite{bib4} the proof of Lemma 3.1 of \cite{bib1} remains the same,
except for the new estimates
$N(\sigma, T) \ll T^{30/13(1-\sigma)}\cdot(\log T)^B$  and
$N^{-1} \le \eta < N^{-1+13/15k\,-\eps}$, which have to be introduced
in (10) of Lemma 1 of \cite{lan16}.
We recall definition \eqref{sec1:2} for the value of $A(N; c)$.

\begin{lemma}
\label{lemma-L2}
Let $\eps$ be an arbitrarily small positive constant, let $k\ge 1$ be
an integer, consider $\phi(n)\in\mathbb{Z}[n]$ such that
$\deg(\phi)=k$ and $\lead(\phi)=1$; let $N$ be a sufficiently large
integer. Then there exists a positive constant
$c_{1}=c_{1}(\eps)$, which does not depend on $\phi$, such that
\[
  \int_{-\xi}^{\xi} | \Etilde_{\phi}(\alpha)|^2 \, \dx \alpha
  :=
  \int_{-\xi}^{\xi}
    \vert \Stilde_{\phi}(\alpha)-\gamma_{k}z^{-1/k} \vert^2 \, \dx\alpha
  \ll_{\phi,\,\eps} N^{2/k\,-1} A(N; - 2c_{1}),
\]
uniformly for $0\le \xi < N^{-1+13/15k\,-\eps}$.
\end{lemma}

\begin{proof}
We introduce the series
\[
  \Stilde_{k,\phi}(\alpha)
  :=
  \sum_{n\ge 1}\Lambda(n)e^{-n^k /N} \e(\phi(n)\alpha),
\]
and recall definition \eqref{sec1:2.5} for $\Stilde_k$.
We decompose
$\tilde{\mathcal{E}}_{\phi}(\alpha):=\tilde{S}_{\phi}(\alpha)-\gamma_{k}z^{1/k}$
into three parts, using the triangle inequality
\begin{align*}
\notag
  \int_{-\xi}^{\xi}
    |\Stilde_{\phi}(\alpha) - \gamma_{k} z^{1/k}|^2 \, \dx \alpha
  &\le
  \int_{-\xi}^{\xi}
    |\Stilde_{k}(\alpha) - \gamma_{k} z^{1/k}|^2\, \dx \alpha \\
\notag
  &\qquad+
  \int_{-\xi}^{\xi}
    |\Stilde_{\phi}(\alpha) - \Stilde_{k,\phi}(\alpha)|^2 \, \dx\alpha \\
\notag
  &\qquad+
  \int_{-\xi}^{\xi}
    |\Stilde_{k,\phi}(\alpha) - \Stilde_{k}(\alpha)|^2 \, \dx \alpha \\
%\label{sec2:7}
  &=:
  \Sigma_{1} + \Sigma_{2} + \Sigma_{3},
\end{align*}
say.
By Lemma 1 of \cite{lan16}, we know that
$\Sigma_{1} \ll_{k} N^{2/k\,-1} A(N;-2c_{1})$ for a suitable constant
$c_1 > 0$.
This is the only place of the proof where we need the restricted range
for $\xi$: from now on we just assume that
$0 \le \xi \le N^{-1 + 1/k \, -\eps}$.

We notice that the remainder of the proof is essentially independent
of the prime numbers.
For $\Sigma_{2}$ and $\Sigma_{3}$, we first need to
split the two differences as
\begin{align}
\label{eq1:lemma-L2}
  \Stilde_{\phi}(\alpha) - \Stilde_{k,\phi}(\alpha)
  &=
  \Bigl(\sum_{n \le M} + \sum_{n > M} \Bigr)
    \Lambda(n) (\e^{-\phi(n)/N} - \e^{-n^{k}/N}) \e(\phi(n)\alpha)
  =:
  S_{2,M}(\alpha) + S_{2,\infty}(\alpha), \\
\label{eq2:lemma-L2}
  \Stilde_{k,\phi}(\alpha) - \Stilde_{k}(\alpha)
  &=
  \Bigl(\sum_{n \le M} + \sum_{n > M} \Bigr)
    \Lambda(n) \e^{-n^{k}/N} (\e(\phi(n)\alpha) - \e(n^k\alpha))
  =:
  S_{3,M}(\alpha) + S_{3,\infty}(\alpha),
\end{align}
say, for suitable values $M = M(N) > 1$, which need not be the same in
both cases.

We write $\eta(n) := \phi(n)-n^k$ so that $d := \deg(\eta)\in [1,k-1]$
by our assumption.
Then we have to apply properly Corollary 2 of Montgomery and Vaughan
\cite{mon74}, for $r=n$, $\lambda_{n}=2\pi\phi(n)$ and
$\delta_{n} := \lambda_{n} - \lambda_{n-1} \gg_{\phi} n^{k-1}$ and partial
summation.
In the following we write $\eta(n) = a_d n^d + \cdots$, where
$a_d \in \Z \setminus \{ 0 \}$.

We recall the decomposition in \eqref{eq1:lemma-L2}.
Our goal is the estimate of
\[
  \Sigma_{2}
  \ll_{\phi}
  \int_{-\xi}^{\xi}
    |S_{2,M}(\alpha)|^2 \, \dx \alpha
  +
  \int_{-\xi}^{\xi}|S_{2,\infty}(\alpha)|^2 \, \dx \alpha.
\]
We choose $M = N^{1 / k} L$ and notice that $M^d = o(N)$, so that
$n^d \le M^d = o(N)$ for all the summands in $S_{2, M}$.
Hence, by the Taylor expansion of the exponential function, we get
\begin{align*}
  \e^{-\phi(n) / N} - \e^{-n^k / N}
  &=
  \e^{-n^k /N} (\e^{-\eta(n) / N} - 1)
  =
  \e^{-n^k /N}
  \Bigl( - \frac{\eta(n)}{N} + \Odig{\frac{\eta^2(n)}{N^2}} \Bigr) \\
  &=
  \e^{-n^k /N}
  \Bigl( - \frac{\eta(n)}{N} + \Odig{\frac{n^{2 d}}{N^2}} \Bigr)
\end{align*}
and therefore we can rewrite $S_{2,M}$ as
\begin{align}
\notag
  S_{2,M}(\alpha)
  &=
  - N^{-1}
  \sum_{n \le M}
    \Lambda(n) \e^{-n^k / N} \eta(n) \e(\phi(n) \alpha)
  +
  \Odig{N^{-2} \sum_{n \le M} \Lambda(n) \, \e^{-n^k/N} n^{2 d}} \\
\label{lemma-5.0}
  &=:
  \Sigma_{2, M, 1} + \Odi{N^{-2} M^{2 d + 1}},
\end{align}
say, by a weak form of the PNT.
Then, by Corollary 2 of \cite{mon74}, we obtain that
\begin{align*}
  \int_{-\xi}^{\xi} |\Sigma_{2, M, 1}(\alpha)|^2 \, \dx \alpha
  &=
  \sum_{n \le M}
    \Lambda(n)^2 \e^{-2n^k /N}\frac{\eta^2(n)}{N^2}
    (2 \xi + \Odim{\delta_{n}^{-1}}) \\
  &\ll
  \xi N^{-2}
  \sum_{n\le M} \Lambda(n)^2 \e^{-2n^k /N} n^{2 d}
  +
  N^{-2}
  \sum_{n \le M}\Lambda(n)^2 \e^{-2n^k /N} n^{2d+1-k}.
\end{align*}
Since $M^d = o(N)$, we have that
\begin{align}
\notag
  \xi N^{-2}
  \sum_{n \le M} \Lambda(n)^2 \e^{-2n^k /N} n^{2d}
  &\ll
  \xi N^{-2} L \sum_{n \le M} \Lambda(n) n^{2 d} \\
\label{lemma-5.1}
  &\ll
  \xi N^{-2} L M^{2 d + 1}.
\end{align}

The $\delta_{n}^{-1}$-term requires to distinguish three cases, the
first being the critical one:

\begin{enumerate}

\item If $2d+1-k \ge 0$, i.e. $2d+1\ge k$, then
\begin{align}
\label{lemma-5.2}
  N^{-2}
  \sum_{n \le M} \Lambda(n)^2 \e^{-2 n^k / N} n^{2 d + 1 - k}
  &\ll
  N^{-2} M^{2 d + 2 - k} L\ll_{\phi}N^{2/k\,-1}A(N;-2c_{1}),
\end{align}
by our choice of $M$.

\item If $2d+1-k=-1$, i.e. $k=2d+2$, then
\[
  N^{-2}\sum_{n \le M} n^{-1} 
  =
  N^{-2}\log M
 \ll_{\phi} N^{2/k\,-1}A(N;-2c_{1}).
\]

\item If $2d+1-k\le -2$, i.e. $k\ge 2d+3$, we have
\[
  N^{-2}\sum_{n \le M} n^{2d+1-k}
  \le
  N^{-2}\sum_{n \le M} n^{-2} \ll N^{-2}\ll_{\phi}N^{2/k\,-1}A(N;-2c_{1}).
\]
\end{enumerate}

Hence, by \eqref{lemma-5.0}, \eqref{lemma-5.1} and~\eqref{lemma-5.2}
we have
\begin{equation*}
  \int_{-\xi}^\xi
    \vert S_{2, M}(\alpha) \vert^2 \, \dx \alpha
  \ll
  \xi N^{-2} L M^{2 d + 1}
  +
  N^{-2} L M^{2 d + 2 - k}
  +
  \xi N^{-4} M^{4 d + 2}\ll N^{2 / k - 1} A(N; -2 c_1),
\end{equation*}
since $M = N^{1 / k} L$.
For $S_{2,\infty}$ we argue as in the evaluation of $S_{\phi,\infty}$ in
Lemma \ref{lem-Sinf} and we find that $M = N^{1/k} L$ is enough for
our purposes.

As before, we bound $\Sigma_{3}$ with two integrals that come from the
decomposition in \eqref{eq2:lemma-L2}:
\[
  \Sigma_{3}
  \ll
  \int_{-\xi}^{\xi} |S_{3,M}(\alpha)|^2 \, \dx \alpha
  +
  \int_{-\xi}^{\xi} |S_{3,\infty}(\alpha)|^2 \, \dx \alpha.
\]
We split the Taylor series of $\e(\eta(n)\alpha)$ into two parts, for
a finite $j_{0}>1$, depending on $\eps$, (thus, from now on, $\ll$
means $\ll_{\eps}$):
\begin{align*}
  S_{3,M}(\alpha)
  &=
  \sum_{n\le M}
    \Lambda(n) \e^{-n^k /N} \e(n^k \alpha) (\e(\eta(n)\alpha) - 1) \\ 
  &=
  \sum_{n\le M}
    \Lambda(n) \e^{-n^k /N} \e(n^k \alpha)
    \left(
      \sum_{j=1}^{j_{0}} \frac{(\eta(n)2\pi i\alpha)^j}{j!}
      +
      \mathcal{O}\left(\frac{\eta(n)^t |2\pi \alpha|^t}{t!}\right)
    \right) \\
  &=:
  (C) + (D),  
\end{align*}
for $t= j_{0}+1$.
We apply the triangle inequality and then Corollary $2$ of
\cite{mon74} to the $L^2$-norm of $(C)$, for every $j\in [1,j_{0}]$:
first we get that the $L^2$-norm of $(C)$ is bounded by
\begin{align*}
  \int_{-\xi}^{\xi}
  &\left| \sum_{j=1}^{j_{0}}
    \sum_{n\le M} \Lambda(n) \e^{-n^k /N} \e(n^k \alpha)
      \frac{(2\pi i\eta(n)\alpha)^j}{j!}\right|^2 \, \dx \alpha \\
  &\ll
  \sum_{j = 1}^{j_0}
    \int_{-\xi}^{\xi}
      \left| \sum_{n\le M} \Lambda(n) \e^{-n^k / N} \e(n^k \alpha)
               \frac{(2\pi i)^j \eta(n)^j \alpha^j}{j!} 
      \right|^2 \, \dx \alpha   \\
  &=:
  C_{1}+\dots+C_{j_{0}}.
\end{align*}
By the Corollary just quoted we get that
\begin{align*}
  C_{j}
  &=
  \int_{-\xi}^{\xi}
    \vert \alpha \vert^{2 j}
    \cdot
    \left|
      \sum_{n\le M}\Lambda(n)
        \e^{-n^k /N} \e(n^k \alpha) \frac{(2\pi i)^j\eta(n)^j}{j!}
    \right|^2 \, \dx \alpha \\
  &\ll
  \xi^{2j+1}\sum_{n\le M} \Lambda(n)^2 \e^{-2n^k /N} n^{2dj}
  +
  \xi^{2j}\sum_{n\le M}\Lambda(n)^2 \e^{-2n^k /N} n^{2dj+1-k} \\
  &=:
  (C_{j,1}) + (C_{j,2}).
\end{align*}
Our choice $M = N^{1/k} L$ and the fact that $d \le k - 1$ imply that
\begin{align*}
  (C_{j,1})
  &\ll
  L^2\xi^{2j+1}M^{2dj+1}
  \le
  L^2 N^{(-1+1/k\,-\eps)(2j+1)} (N^{1/k}L)^{2dj+1} \\ 
  &\ll
  N^{2/k\,-1-\eps-2j\eps} L^{2 j (k - 1) + 3} 
  \ll
  N^{2/k\,-1}A(N;-2c_{1})
\end{align*}
since $\eps > 0$. 
For $(C_{j,2})$ we need only consider the case when $d$ is large, so
that $2dj+1-k\ge 0$; we find that
\begin{align*}
  (C_{j,2})
  &\ll
  \xi^{2j} L^2 M^{2dj+2-k}
  \le
  N^{-2j+2j/k\,-2j\eps}L^2 (N^{1/k}L)^{2kj-2j+2-k} \\  
  &\ll
  N^{-2j\eps+2/k\,-1} L^{2kj-2j+4-k}
  \ll
  N^{2/k\,-1}A(N;-2c_{1}),
\end{align*}
as above.
Therefore, for every $1\le j\le j_{0}$, we can estimate $C_{j}$.

Now we compute and bound the $L^2$-norm of $(D)$:
\begin{align*}
  \int_{-\xi}^{\xi}
  &
    \left|\sum_{n\le M}\Lambda(n) \e^{-n^k/N} \e(n^k\alpha)
    \mathcal{O}\left(\frac{\eta(n)^t |2\pi\alpha|^t}{t!}\right) \right|^2
  \, \dx \alpha \\
  &\ll
  \int_{-\xi}^{\xi}
    |\alpha|^{2t}\left|\sum_{n\le M}\Lambda(n) \e^{-n^k/N}(2\pi)^t\eta(n)^t\right|^2
  \, \dx \alpha \\
  &\ll
  \xi^{2t+1}\left|\sum_{n\le M}\Lambda(n) \e^{-n^k/N}n^{dt}\right|^2
  \ll
  \xi^{2t+1}M^{2(dt+1)} L^2
  \le
  N^{(-1+1/k\,-\eps)(2t+1)}(N^{1/k}L)^{2dt+2} \\ 
  &\ll
  N^{-(2t+1) \eps - 1 + 3/k} L^{2 (k - 1) t + 4}.
\end{align*}
In order to conclude the proof of our Lemma we need to show that this
is $\ll N^{2/k\,-1} A(N;-2c_{1})$.
We choose, as we may, $t = j_0 + 1$ so large that $(2 t + 1) \eps > 2 / k$.
\end{proof}

\begin{lemma}
\label{lemma-tolev}
For $k>1$, $B=N^{2\eps}$ and $\xi_{0}:=N^{-1+13/15k\,-\eps}$, we have
\begin{align}
\int_{-\tau}^{\tau}|\Stilde_{\phi}(\alpha)|^2\,\dx\alpha \ll_{\phi}
\begin{cases}
 B N^{2/k\,-1}A(N;-2c_{1}) & \text{if \,$0 \le \tau < \xi_{0}$} \\
 \tau N^{1/k}L & \text{if \,$0\le \tau \le \frac{1}{2}$}.
\end{cases}
\end{align}
\end{lemma}

\begin{proof}
If $0\le \tau < \xi_{0}$ by the triangle inequality, adding and subtracting $\gamma_{k}z^{-1/k}$, we have
\begin{align*}
\int_{-\tau}^{\tau}|\Stilde_{\phi}(\alpha)|^2\,\dx\alpha &\le 2\int_{-\tau}^{\tau}|\Etilde_{\phi}(\alpha)|^2\,\dx\alpha + 2\gamma_{k}^2\int_{-\tau}^{\tau}|z|^{-2/k}\,\dx\alpha.
\end{align*}
By previous Lemma \ref{lemma-L2}, we know that the first summand is $\ll_{\phi}N^{2/k\,-1}A(N;-2c_{1})$, thus we have to show that 
\begin{align}
\label{int-z}
\int_{-\tau}^{\tau}|z|^{-2/k}\,\dx\alpha\ll B N^{2/k\,-1}A(N;-2c_{1}).
\end{align}
Since $|z|^2 = N^{-2} + (2\pi \alpha)^2$, by the change of variable $x=2\pi\alpha N$, we can rewrite the left hand side of \eqref{int-z} as
\begin{align*}
N^{2/k\,-1}\int_{-2\pi\tau N}^{2\pi\tau N} (1+x^2)^{-1/k}\,\dx x = 2N^{2/k\,-1}\int_{0}^{2\pi\tau N}(1+x^2)^{-1/k}\,\dx x,
\end{align*}
which is $\ll_{k} 1+\log(2\pi \tau N)$, if $k=2$. We may therefore neglect this contribution, in view of the bound we want to achieve. For $k\ge 3$, we find that
\begin{align*}
\int_{0}^{2\pi\tau N}(1+x^2)^{-1/k}\,\dx x \ll 1+\int_{1}^{2\pi\tau N}x^{-2/k}\,\dx x \ll_{k} (2\pi\tau N)^{1-2/k} + 1,
\end{align*}
Therefore, the l.h.s. of \eqref{int-z} is 
\begin{align*}
N^{2/k\,-1}\int_{-2\pi\tau N}^{2\pi\tau N} (1+x^2)^{-1/k}\,\dx x&\ll_{k}N^{2/k\,-1}( N^{1-2/k}\tau^{1-2/k} +1)=\tau^{1-2/k} + N^{2/k\,-1} \\ &=(BH^{-1})^{1-2/k} + N^{2/k\,-1}\ll B N^{2/k\,-1} A(N;-2c_{1}) 
\end{align*}
since $(H/N)^{2/k\,-1}\ll B^{2/k} A(N;-2c_{1})$.
For $0 \le \tau \le 1/2$, the proof relies on Corollary~2 of
\cite{mon74} and is essentially the same as the proof of Lemma~2 of
\cite{lan16}.
\end{proof}

\begin{remark}
\label{tolev}
With respect to \cite{bib1}, we do not evaluate the $L^2$-norm of $\Stilde_{\phi}$ via Lemma~7 of Tolev \cite{bib11}; in the notation of Lemma~\ref{lem-Sinf}, we are now dealing with the finite sum $\Stilde_{\phi,R}:=\Stilde_{\phi}|_{[1,R]}$, instead of $\Stilde_{\phi}|_{[\delta R, R]}$, for $\delta>0$ and $R=R(N)$. As a consequence, in the sub-interval $[1,\delta R]$, we cannot ensure an upper bound depending on $R$, for $\int_{-\tau}^{\tau}|\Stilde_{\phi,R}(\alpha)|^2\,\dx\alpha$. 
\end{remark}

\begin{lemma}[Lemma~4 of \cite{lan16jnt}]
\label{lemma-mt}
Let $N$ be a positive integer, $z = z(\alpha) = N^{-1} - 2 \pi i \alpha$
and $\mu > 0$.
Then, uniformly for $n \ge 1$ and $X > 0$ we have
\[
  \int_{-X}^X z^{-\mu} \e(- n \alpha) \, \dx \alpha
  =
  \e^{- n / N} \frac{n^{\mu - 1}}{\Gamma(\mu)}
  +
  \mathcal{O}_{\mu} \Bigl( \frac1{n X^{\mu}} \Bigr).
\]
\end{lemma}

\section{Proof of the main result}

\subsection{Proof of Theorem~\ref{thm-unc}}

Recalling definitions \eqref{sec1:1} and \eqref{sec1:2} and formulas \eqref{I123}, we are ready
to show our result.
The evaluation of $I_{1}$ directly follows from Lemma~\ref{lemma-mt},
with $\mu = j / k$, $X = \tau = B / H$: in fact we have
\[
  I_{1}
  =
  \int_{-\tau}^{\tau} \frac{U(-\alpha, H) \e(-N\alpha)} {z^{j/k}} \, \dx \alpha
  =
  \frac{1}{\gamma_{k,j}}\sum_{n=N+1}^{N+H} \e^{-n/N}n^{(j-k)/k}
  +
  \mathcal{O}_{k,j}\biggl(\frac{H}{N}\biggl(\frac{H}{B}\biggr)^{j/k}\biggr).
\]
By Lemma~3.6 of \cite{bib1}, we have that
\[
  \sum_{n=N+1}^{N+H} \e^{-n/N}n^\lambda
  =
  \e^{-1} H N^\lambda + \mathcal{O}_{\lambda}(H^2 N^{\lambda-1}),
\]
therefore, with $\lambda=(j-k)/k$, we have that
\[
  I_{1}
  =
  \frac{1}{\e \gamma_{k,j}} H N^{(j-k)/k}
  +
  \mathcal{O}\biggl(H^2N^{(j-k)/k\,-1} + \frac{H}{N}\biggl(\frac{H}{B}\biggr)^{j/k}\biggr)
\]
and the second summand in the error term can be neglected, since
$H\le N$ and $B>N^{\eps}$.

For $I_{2}$, we need Lemma~\ref{lemma-L2}, the estimate
\[
  |z|^{-1} \ll \min\{N,\,|\alpha|^{-1}\}
\]
and the identity (2) of \cite{bib1} - that we recall below and that is
applied with $x=\Stilde_{\phi}(\alpha)$, $y=\gamma_{k}z^{-1/k}$ and
$j\ge k\ge 2$:
\begin{equation}
\label{sec3:8}
  x^{j}-y^{j}
  =
  (x - y)^2 \sum_{\ell=1}^{j-1} \bigl(\ell x^{j-1-\ell}y^{\ell-1}\bigr)
  +
  j (x - y) y^{j - 1}.
\end{equation}
From $\eqref{sec3:8}$, Lemma \ref{lemma-L2} and the Cauchy-Schwarz inequality, as in (12) of \cite{bib1}, it follows that
\begin{align}
\notag
  I_{2}
  &\ll_{\phi} 
  H 
  \Bigl(
    \sum_{\ell=1}^{j-1}
      \int_{-B/H}^{B/H}|x - y|^2 |x|^{j-1-\ell}|y|^{\ell-1}\,\dx\alpha 
  \Bigr)
  +
  H \int_{-B/H}^{B/H} |x-y||y|^{j-1}\,\dx\alpha \\ \notag
&\ll_{\phi}H\sum_{\ell=1}^{j-1}\max_{\alpha} |x|^{j-1-\ell}\max_{\alpha}|y|^{\ell-1}\int_{-B/H}^{B/H}|x - y|^2 \,\dx\alpha \\ \notag
&\qquad+H\left(\int_{-B/H}^{B/H}|x - y|^2 \,\dx\alpha\right)^{1/2}\left(\int_{-B/H}^{B/H}|y|^{2(j-1)} \,\dx\alpha\right)^{1/2} \\ \notag
&\ll_{\phi}H\sum_{\ell=1}^{j-1} N^{(j-1-\ell)/k}\cdot N^{(\ell -1)/k} N^{2/k\,-1}A(N;-2c_{1}) \\ \notag
&\qquad + HN^{1/k\,-1/2}A(N;-c_{1})\left(\int_{-B/H}^{B/H}|z|^{-2(j-1)/k} \,\dx\alpha\right)^{1/2} \\ 
\notag
&\ll_{\phi} H N^{j/k\,-1} A(N;-2c_{1}) \\ \label{est-I_2-1}
&\qquad+ HN^{1/k\,-1/2} A(N;-c_{1})\cdot N^{(j-1)/k\,-1/2}\left(\int_{-2\pi N B/H}^{2\pi N B/H}\frac{\dx x}{(1+x^2)^{(j-1)/k}}\right)^{1/2} \\ \label{est-I_2-2}
&\ll_{\phi}H N^{j/k\,-1} A(N;-2c_{1}) + HN^{j/k\,-1}A(N;-c_{1}),
\end{align} 
if $j\ge k\ge 2$ and $j\ne 2$, since the integral in \eqref{est-I_2-1} converges and its contribution is equal to $\mathcal{O}_{k,j}(1)$. For $j=k=2$, the summand in \eqref{est-I_2-1} is $\ll_{\phi} HA(N;-c_{1})\cdot N^{j/k\,-1}L^{1/2}$.
Thus, by \eqref{est-I_2-2}, we infer that
\[
  I_{2}
  \ll_{\phi}
  H N^{j/k\,-1} A(N; - c_1),
\]
where $c_{1}(\eps)>0$ is the constant in Lemma~\ref{lemma-L2},
provided that $B H^{-1} < N^{-1+13/15k\,-\eps}$, that is,
$H > N^{1 - 13 / 15 k + \eps}$.
For the evaluation of $I_{3}$ we set
\[
  F(\xi)
  :=
  \int_{0}^{\xi}|\Stilde_{\phi}(\alpha)|^2 \, \dx \alpha.
\]
By partial integration and Lemma~\ref{lemma-tolev}, for
$\tau = B / H$, we have
\begin{align*}
  \int_{\mathcal{C}}|\Stilde_{\phi}(\alpha)|^2\,\frac{\dx \alpha}{\alpha}
  &\ll_{\phi}
  \frac{F(\alpha)}{\alpha}\bigg|_{\tau}^{1/2}
  +
  \int_{\tau}^{1/2} \frac{F(\alpha)}{\alpha^2} \, \dx \alpha \\
  &\ll_{\phi}
  \tau^{-1}F(\tau) + \int_{\tau}^{1/2}\frac{\alpha N^{1/k}L}{\alpha^2}\,\dx\alpha\\ 
  &\ll_{\phi} HN^{2/k\,-1}A(N;-2c_{1}) + N^{1/k}L^2 \\ &\ll_{\phi}  HN^{2/k\,-1}A(N;-2c_{1}),
\end{align*}
since 
\[
HN^{2/k\,-1}> N^{1/k} \iff H> N^{1-1/k},
\]
which is consistent with our choice of $H$.
Then, recalling that $j \ge 2$ and $|U(\alpha, H)|\ll\min\{H, |\alpha|^{-1}\}$,
we obtain
\begin{align*}
  I_{3}
  &\ll_{\phi}
  \max_{\alpha\in[-1/2,1/2]} |\Stilde_{\phi}(\alpha)|^{j-2}
  \int_{\mathcal{C}}|\Stilde_{\phi}(\alpha)|^2 \, \frac{\dx \alpha}{\alpha} \\
  &\ll_{\phi}
  N^{(j-2)/k}\cdot HN^{2/k\,-1}A(N;-2c_{1}) \\ 
  &=  HN^{j/k\,-1}A(N;-2c_{1}).
\end{align*}
Summing the three terms $\gamma_k^j I_1$, $I_2$ and $I_3$ (defined in \eqref{I123}) we arrive at
\begin{align}
\notag
  \sum_{n=N+1}^{N+H} \e^{-n/N}R_{\phi,j}(n)
  &=
  \frac{\gamma_{k}^j}{\gamma_{k,j}} \cdot \, \e^{-1}HN^{(j-k)/k} \\
\label{final-est}
  &\qquad+
  \mathcal{O}\bigl(N^{(j-k)/k} H \cdot A(N; -c_1)\bigr).
\end{align}
To complete the proof, we remove the exponential factor at the
left-hand side: in fact, as $\e^{-n/N} \in [\e^{-2}, \e^{-1}]$,
\begin{align*}
  \e^{-2} \sum_{n=N+1}^{N+H} R_{\phi,j}(n)
  &\le
  \sum_{n=N+1}^{N+H} \e^{-n/N}R_{\phi,j}(n) \ll_{\phi} HN^{(j-k)/k}
\end{align*}
so that
\[
  \sum_{n=N+1}^{N+H} R_{\phi,j}(n)
  \ll_{\phi} H N^{(j-k)/k}.
\]
We use this inequality for an estimate of the error term that appears
from the development
$\e^{-n/N} = \e^{-1-(n-N)/N} = \e^{-1}(1+\mathcal{O}((n-N)N^{-1}))$:
\begin{align*}
  \sum_{n=N+1}^{N+H} \e^{-n/N}R_{\phi,j}(n)
  &=
  \sum_{n=N+1}^{N+H} \e^{-1}(1+\mathcal{O}((n-N)N^{-1}))R_{\phi,j}(n) \\
  &=
  \e^{-1} \sum_{n=N+1}^{N+H} R_{\phi,j}(n)
  +
  \mathcal{O}_{\phi}(H^2N^{(j-k)/k\,-1}),
\end{align*}
and Theorem~\ref{thm-unc} follows from~\eqref{final-est} since we
assumed that $H < N^{1 - \eps}$ so that the last error term above is
smaller than the one in~\eqref{final-est}.

\subsection{The case \texorpdfstring{$\lead(\phi)>1$}{lead}}

If we slightly perturb the previous case, removing the hypothesis
$\lead(\phi) = a_{k} = 1$, we get the new variable $w=a_{k}z$, instead of
$z$. It implies that the main term of Theorem~\ref{thm-unc} is now
multiplied by $a_{k}^{-j/k}$, for $j\ge k$, while $I_{3}$ is estimated
as in Lemma~\ref{bound-S}, thanks to a similar change of variables
$u=a_{k}t^k /N$, in the integral term.

In order to bound $I_{2}$, we have to verify that Lemma~\ref{lemma-L2}
 keeps holding. That's what we are going to show in the following result.

\begin{lemma}
Let $\eps$ be an arbitrarily small positive constant, $k\ge 1$ be an
integer and $a_{k}\in\mathbb{N}^*$. Then there exists a positive
constant $c_{1}=c_{1}(\eps)$, which does not depend on $k$, such that
\[
  \int_{-\xi}^{\xi}
    |\Stilde_{\phi}(\alpha)-(a_{k}\gamma_{k})^{-1/k}z^{-1/k}|^2 \,
    \dx \alpha \ll_{\phi} N^{2/k\,-1} A(N; -c_{1}),
\]
uniformly for $0\le\xi < N^{-1+13/15k\,-\eps}$.
\end{lemma}

\begin{proof}
We first restrict to $\phi(n)=a_{k}n^k$; in fact, this case is
straightforward from Lemma 1 of \cite{lan16}, replacing $z$ with
$w=a_{k}z$ and checking that the constant factor $a_{k}$ in
\[
  w^{-\rho/k}
  =
  (a_{k}|z|)^{-\rho/k}\exp(-i(\rho/k)
  \arctan{2\pi N a_{k}\alpha}-\pi |\gamma|/2k)
\]
does not affect the estimate of Lemma 1 of \cite{lan16}. Then,
mimicking the strategy of Lemma \ref{lemma-L2}, we extend the result
to a generic polynomial $\phi\in\mathbb{Z}[n]$.
\end{proof}

\section{Acknowledgements}
\noindent We thank the referee for pointing out several inaccuracies in the first version of this paper. 

\providecommand{\bysame}{\leavevmode\hbox to3em{\hrulefill}\thinspace}
\providecommand{\MR}{\relax\ifhmode\unskip\space\fi MR }
% \MRhref is called by the amsart/book/proc definition of \MR.
\providecommand{\MRhref}[2]{%
  \href{http://www.ams.org/mathscinet-getitem?mr=#1}{#2}
}
\providecommand{\href}[2]{#2}

\end{document}